\documentclass[reqno]{amsart}
\usepackage{amsthm}
\usepackage{amsmath}
\usepackage{amssymb}
\usepackage{amsfonts}
\usepackage{txfonts}
\usepackage[margin=1.5in]{geometry}
\usepackage{hyperref}

\newtheorem{theorem}{Theorem}[section]
\newtheorem{definition}[theorem]{Definition}
\newtheorem{corollary}[theorem]{Corollary}
\newtheorem{proposition}[theorem]{Proposition}
\newtheorem{conjecture}[theorem]{Conjecture}
\newtheorem{lemma}[theorem]{Lemma}
\newtheorem{remark}[theorem]{Remark}
\newtheorem{question}[theorem]{Question}

\DeclareMathOperator{\Alb}{Alb}
\DeclareMathOperator{\Stab}{Stab}
\DeclareMathOperator{\NS}{NS}
\DeclareMathOperator{\Pic}{Pic}
\DeclareMathOperator{\GL}{GL}
\DeclareMathOperator{\ID}{ID}
\DeclareMathOperator{\Tr}{Tr}
\DeclareMathOperator{\Tors}{Tors}
\DeclareMathOperator{\Preper}{Preper}
\DeclareMathOperator{\End}{End}

\author{Holly Krieger and Paul Reschke}
\date{\today}
\title{Cohomological Conditions on Endomorphisms of Projective Varieties}

\thanks{The first author was partially supported by NSF grant DMS-1303770. The second author was partially supported by NSF grants DMS-0943832 and DMS-1045119}

\begin{document}

\maketitle

\begin{abstract}
We characterize possible periodic subvarieties for surjective endomorphisms of complex abelian varieties in terms of the eigenvalues of the cohomological actions induced by the endomorphisms, extending previous work in this direction by Pink and Roessler \cite{PinRoe}. By applying our characterization to induced endomorphisms of Albanese varieties, we draw conclusions about the dynamics of surjective endomorphisms for a broad class of projective varieties. We also analyze several classes of surjective endomorphisms that are distinguished by properties of their cohomological actions.
\end{abstract}

\section{Introduction}

In this note, we study the nature of periodic subvarieties for endomorphisms of smooth complex projective varieties. The starting point for our investigation is a theorem due to Pink and Roessler:

\begin{theorem}[\cite{PinRoe}, Theorem 2.4]\label{PinRoe}
Let \(f: A \rightarrow A \) be an isogeny of a complex abelian variety \(A\), and suppose that no eigenvalue of $f^* |_{H^{1,0}(A)}$ is a root of unity. Suppose that \(V \subseteq A\) is a reduced and irreducible subvariety satisfying \(f(V)=V\). Then \(V\) is a translate of an abelian subvariety of \(A\).
\end{theorem}

By the Lefschetz Fixed-Point Theorem, the eigenvalue condition in Theorem \ref{PinRoe} guarantees that $f$ has a fixed point, and therefore is conjugate by a translation to an isogeny, even if $f$ is only assumed to be a surjective endomorphism (i.e., not necessarily a homomorphism).  Thus the conclusion holds for any surjective endomorphism $f$ satisfying the eigenvalue condition. (See \S \ref{Eigenvals} and \S \ref{UnityFree1} below.) We will not assume in the following that a surjective endomorphism of an abelian variety is an isogeny. 

We extend Theorem \ref{PinRoe} to the case where \(f^*\) may have eigenvalues on \(H^{1,0}(A)\) that are roots of unity; here, $\kappa(V)$ denotes the Kodaira dimension of any smooth birational model of a variety $V$:

\begin{theorem}\label{MainThm}
Let \(f\) be a surjective endomorphism of a complex abelian variety \(A\), and suppose that \(V \subseteq A\) is a reduced and irreducible subvariety satisfying \(f(V)=V\). Then there is a reduced and irreducible subvariety \(W \subseteq V\) with \(\kappa(W)=\dim(W)=\kappa(V)\), and some iterate $f^k$, such that \(V=\Stab^0_A(V)+W\) and \(f^k(\Stab^0_A(V)+w)=\Stab^0_A(V)+w\) for every \(w \in W\).
\end{theorem}

The proof of Theorem \ref{MainThm} has a similar flavor to the proof of Theorem \ref{PinRoe} by Pink and Roessler: by Ueno \cite{Uen1}, all subvarieties of \(A\) can be built from tori and varieties of general type; we then apply Kobayashi and Ochiai \cite{KobOch}, which states that every rational self-map of a variety of general type has finite order. (See \S \ref{CharSubvars} and \S \ref{DomRatMaps} below.) Note that \(W\) may be singular or zero-dimensional; in particular, if \(\kappa(V)=0\), then \(V\) is a translate of an abelian subvariety of \(A\). 

As a corollary of Theorem \ref{MainThm}, we recover the following mild strengthening of the theorem of Pink and Roessler:

\begin{corollary}\label{PinRoeCor} Let \(f\) be a surjective endomorphism of a complex abelian variety \(A\), and suppose that \(V \subseteq A\) is a reduced and irreducible subvariety satisfying \(f(V)=V\). Let $u_f$ denote the number of root-of-unity eigenvalues of $f^* |_{H^{1,0}(A)}$ with multiplicity.  Then 
$$\kappa(V) \leq u_f;$$
in fact, the inequality is strict except possibly if \(\kappa(V)=0\).
\end{corollary}

Suppose now that \(X\) is an arbitrary smooth complex projective variety, and that \(f\) is a surjective endomorphism of \(X\). Since the Albanese variety \(\Alb(X)\) is generated by the image of \(X\) under the Albanese map \(a_X\), \(f\) induces a surjective endomorphism \(F\) of \(\Alb(X)\); moreover, if \(a_X(X) \neq \Alb(X)\), then \(a_X(X)\) is a reduced and irreducible proper subvariety of \(\Alb(X)\) satisfying \(F(a_X(X))=a_X(X)\). (See \S \ref{AlbIntro} below.) So we can use Theorem \ref{MainThm} to draw conclusions about endomorphisms of varieties with non-surjective Albanese maps, as in:

\begin{corollary}\label{PerSubCor}
Let $X$ be a smooth complex projective variety with \(a_X(X) \neq \Alb(X)\), and suppose that \(f\) is an infinite-order surjective endomorphism of $X$. Then there is a proper positive-dimensional subvariety of \(X\) that is periodic for \(f\).
\end{corollary}

Note that any variety $X$ in Corollary \ref{PerSubCor} must have \(\kappa(a_X(X))>0\); but it will follow from the proof that this is not necessarily true for the periodic subvariety, so this corollary cannot be used for induction. Note also that a smooth curve with a non-surjective Albanese map is necessarily hyperbolic and hence, by the De Franchis Theorem, does not admit any infinite-order endomorphisms. If we write \(a_X(X)=B+W\), where \(B\) is the stabilizer of \(a_X(X)\) in \(\Alb(X)\) and \(W\) has \(\kappa(W)=\dim(W)\), then the periodic subvariety in Corollary \ref{PerSubCor} can be taken to be the pre-image under \(a_X\) of \(B+w\) for any \(w \in W\). (See \S \ref{AlbCors} below.) 

We turn now to an assessment of certain classes of endomorphisms that are characterized by cohomological properties.

\begin{definition}[\cite{NakZha},\cite{Zha}]\label{Polarized}
Let \(f\) be a surjective endomorphism of a projective variety \(X\). We say that \(f\) is \textbf{polarized} if there is an ample line bundle \(L \in \Pic(X)\) such that \(f^*(L)=L^{\otimes q}\) for some integer \(q > 1\).
\end{definition}

The study of polarized endomorphisms - and those varieties which carry them - is of particular interest to dynamicists.  For an endomorphism $f$ of a complex variety $X$, Fakhruddin \cite{Fak} showed that the condition that $f$ is polarized is equivalent to the existence of an embedding $i:X \rightarrow \mathbb{P}^N$ and a morphism $F: \mathbb{P}^N \rightarrow \mathbb{P}^N$ so that $i \circ f = F \circ i$; in the case where $f$ is defined over a field with arithmetic (i.e. a number field or function field), Call and Silverman \cite{CalSil} showed that polarization implies the existence of a dynamical canonical height function on $X$, an important tool in arithmetic dynamics. 

If \(f\) is a polarized endomorphism, then the ample line bundle \(L\) satisfying \(f^*(L)=L^{\otimes q}\) also has the property that \(f^*(L) \otimes L^{-1}\) is ample, which leads to a generalization of the notion of polarization.

\begin{definition}\label{Amplified}
Let \(f\) be a surjective endomorphism of a projective variety \(X\). We say that $f$ is \textbf{amplified} if there is a line bundle \(L \in \Pic(X)\) such that \(f^*(L) \otimes L^{-1}\) is ample.
\end{definition}

Note that the line bundle $L$ in Definition \ref{Amplified} need not itself be ample in general. Suppose that \(f\) is an amplified endomorphism of a variety \(X\). A theorem due to Fakhruddin \cite{Fak} states that the set of periodic points for \(f\) is Zariski dense in $X$. If \(V\) is a subvariety of \(X\) satisfying \(f(V)=V\), then the restriction of \(f\) to \(V\) is again amplified; so the periodic points for \(f\) include a Zariski dense subset of \(V\) as well.  In particular, amplified endomorphisms satisfy one direction of a dynamical version of the Manin-Mumford conjecture; indeed, we observe that the class of endomorphisms for which a dynamical Manin-Mumford conjecture can be made sensible is likely to be precisely the class of amplified endomorphisms.  (See \S \ref{DMM} and \S \ref{BasicProps} below.) Thus the study of varieties carrying amplified endomorphisms is again of dynamical interest.

\begin{definition}\label{Unity-Free}
Let \(f\) be a surjective endomorphism of a projective variety \(X\) with \(\Alb(X) \neq \{0\}\). We say that $f$ is \textbf{unity-free} if no eigenvalue of $f^* |_{H^{1,0}(X)}$ is a root of unity.
\end{definition}

Note that the condition \(\Alb(X) \neq \{0\}\) in Definition \ref{Unity-Free} is equivalent to the condition \(H^{1,0}(X) \neq \{0\}\) (so that \(f^*\) has at least one eigenvalue on \(H^{1,0}(X)\)). (See \S \ref{AlbIntro} below.) The hypothesis on \(f\) in Theorem \ref{PinRoe} is that it is unity-free. In this case, the conclusion that \(f\) has a Zariski dense set of periodic points in every periodic subvariety follows from Lemma \ref{EigenLemma} and Proposition \ref{MainProp} below - without the requirement that $f$ is amplified. (See \S \ref{DMM} below.) 

Theorem \ref{MainThm} restricts the set of varieties which admit unity-free endomorphisms:

\begin{corollary}\label{PolEndCor} Let $X$ be a smooth complex projective variety with \(a_X(X) \neq \Alb(X)\). Then \(X\) does not admit a unity-free endomorphism.
\end{corollary}

Note that the Albanese map for a variety is non-trivial if it is non-surjective, so that it makes sense to speak of unity-free endomorphisms in this setting. Corollary \ref{PolEndCor} complements work by Dinh, Nguyen, and Truong \cite{DinNguTru} which under the same hypotheses shows that \(X\) does not admit an endomorphism with distinct consecutive dynamical degrees. The condition that an endomorphism has distinct consecutive dynamical degrees is disjoint from the condition that it is unity-free; on the other hand, it follows from a theorem of Serre \cite{Ser} that every polarized endomorphism satisfies both of these conditions. (See \S \ref{AlbCors} below.)

Fakhruddin \cite{Fak} showed that any variety admitting a polarized endomorphism must have non-positive Kodaira dimension, and a theorem due to Kawamata \cite{Kaw} states than any variety whose Kodaira dimension is zero has a surjective Albanese map; however, there are many examples of varieties with negative Kodaira dimension and non-surjective Albanese maps. In particular, Corollary \ref{PolEndCor} applies to any bundle over a variety whose Albanese map is non-surjective; in this way it is a generalization of the observation by S. Zhang \cite{Zha} that a ruled surface over a hyperbolic curve cannot admit a polarized endomorphism. (See \S \ref{AlbCors} below.) We note also that work by Nakayama and D.-Q. Zhang \cite{NakZha} offers further characterizations of varieties admitting polarized endomorphisms.

We show that Corollary \ref{PolEndCor} applies to amplified endomorphisms as well:

\begin{theorem}\label{EndoThm}
Let \(f\) be a surjective endomorphism of a smooth complex projective variety \(X\) with \(\Alb(X) \neq \{0\}\). If \(f\) is amplified, then it is unity-free.
\end{theorem}

Theorem \ref{EndoThm} yields the following implication diagram for a surjective endomorphism of a smooth complex projective variety whose Albanese map is non-trivial:
\vspace{.05in}
\[
\text{polarized} \implies \text{amplified} \implies \text{unity-free} \implies \text{infinite-order}.
\]

None of the reverse implications in the diagram hold in general; however, we show that every unity-free endomorphism of an abelian surface is amplified, and we speculate that the same may be true on any abelian variety. (See \S \ref{RevImp} below.) We observe that the failure of an endomorphism to be amplified indicates that the endomorphism must fix the numerical equivalence class of some line bundle - and so is similar to the failure of an endomorphism to be unity-free. (See \S \ref{BasicProps} below.)

The conclusion of Corollary \ref{PolEndCor} for amplified endomorphisms is new and provides a constraint on the types of varieties that should constitute a natural arena for a dynamical Manin-Mumford conjecture. We observe that the set of varieties admitting amplified endomorphisms is strictly larger than the set admitting polarized endomorphisms; for example, a K3 surface may admit an amplified endomorphism but can never admit a polarized endomorphism. (See \S \ref{RevImp} below.)

While an infinite-order endomorphism of a projective variety \(X\) need not induce an infinite-order endomorphism of \(\Alb(X)\), every unity-free endomorphism of a projective variety \(X\) does induce a unity-free endomorphism of \(\Alb(X)\). To the end of better understanding the relationships between polarized, amplified, and unity-free endomorphisms in general, we ask:

\begin{question}\label{2ndQ}
Does an amplified (resp. polarized) endomorphism of a projective variety \(X\) with $\Alb(X) \neq \{0\}$ induce an amplified (resp. polarized) endomorphism of \(\Alb(X)\)?
\end{question}

Note that Question \ref{2ndQ} has an affirmitive answer for amplified endomorphisms if it is true that every unity-free endomorphism of an abelian variety is amplified. 
\\
\\
\noindent The remainder of this paper is organized as follows: in \S 2, we prove Theorem \ref{MainThm} and Corollary \ref{PinRoeCor}, and we give a useful characterization of unity-free endomorphisms of abelian varieties; in \S 3, we prove Corollary \ref{PerSubCor} and Corollary \ref{PolEndCor}, we present results about varieties admitting polarized endomorphisms, and we discuss the notion of a dynamical Manin-Mumford problem; in \S 4, we present a variety of facts about polarized, amplified, and unity-free endomorphisms and the connections between them, and we prove Theorem \ref{EndoThm}.
\\
\\
\noindent \textbf{Acknowledgements.} We thank Mihnea Popa, Laura DeMarco, and Ramin Takloo-Bighash for several useful discussions in the early stages of writing this paper.  
 
\section{Invariant Subvarieties for Endomorphisms of Abelian Varieties}

We say that a subvariety of an abelian variety is an abelian subvariety if it is a reduced and irreducible group subvariety. Given a complex abelian variety \(A\) and a reduced and irreducible subvariety \(V \subseteq A\), the stabilizer of \(V\) in \(A\) is the reduced (but not necessarily irreducible) group subvariety
\[\Stab_A(V) = \{a \in A \mid a + V = V\} \subseteq A.\]
The connected component of \(\Stab_A(V)\) containing the identity is an abelian subvariety \(\Stab^0_A(V) \subseteq A\). (See, e.g., \cite{Hin}.) Given an endomorphism \(f:A \rightarrow A\), there is a homomorphism \(\phi_f:A \rightarrow A\) (which is an isogeny if \(f\) is surjective) and an element \(\tau_f \in A\) such that
\[f(a)=\phi_f(a)+\tau_f\]
for every \(a \in A\). (See, e.g., \cite{Mum}, \S II.) If \(f(V)=V\), then \(\phi_f(\Stab^0_A(V)) \subseteq \Stab^0_A(V)\) (with equality if \(f\) is surjective).

For a complex projective variety \(X\), we let \(\kappa(X)\) denote the Kodaira dimension of any smooth birational model of \(X\) and we say that \(X\) (whether it is smooth or not) is a variety of general type if \(\kappa(X)=\dim(X)\). (Compare, e.g., \cite{Uen2} and \cite{Kob}, \S 7.)

\subsection{Characterization of Subvarieties of Abelian Varieties}\label{CharSubvars}

\begin{theorem}\label{UenoPlus}
Let \(V\) be a reduced and irreducible subvariety of a complex abelian variety \(A\). Then \(V=\Stab^0_A(V)+W\) for some reduced and irreducible subvariety \(W \subseteq V\) with \(\kappa(W)=\dim(W)=\kappa(V)\).
\end{theorem}

\begin{proof}
Set \(B = \Stab^0_A(V)\). By the proof of Theorem 3.10 in \cite{Uen1}, there is an abelian subvariety \(B' \subseteq A\) such that the quotient map \(q':A \rightarrow A/B'\) gives \(V\) the structure of a fiber bundle whose base \(q'(V)\) has \(\kappa(q'(V))=\dim(q'(V))\) and whose fibers are isomorphic to \(B'\). Since each fiber of \(q'\) is invariant under the action of \(B'\), we have \(B' \subseteq B\). The quotient map \(q:A \rightarrow A/B\) satisfies \(q^{-1}(q(V))=V\), and hence also gives \(V\) the structure of a fiber bundle (over \(q(V)\) with fibers isomorphic to \(B\)). Since \(\Stab^0_{A/B'}(q'(V))=q'(B)\), \(q'(V)\) is a fiber bundle over \(q(V)\) with fibers isomorphic to \(q'(B)\). Thus we must have
\[\dim(q'(V)) \geq \dim(q(V)) \geq \kappa(q'(V)) = \dim(q'(V)).\]
(See, e.g., \cite{Uen2}.) So \(\dim(q'(B))=0\), \(B'=B\), and \(q'=q\).

By the Poincar\'e Reducibility Theorem, there is an abelian subvariety \(T \subseteq A\) such that \(A=B+T\) and the addition map from \(B \times T\) to \(A\) is an isogeny. (See, e.g., \cite{Hin}.) So the restriction \(q:T \rightarrow A/B\) is an isogeny, and we can choose an irreducible component \(W \subseteq V \cap T\) so that \(q:W \rightarrow q(V)\) is a finite surjective morphism. Then
\[\dim(q(V))=\dim(W) \geq \kappa(W) \geq \kappa(q(V)) = \dim(q(V)).\]
(See, e.g., \cite{Uen2}.) Since \(q(W) = q(V)\), every \(w \in W\) can be written as \(w=b+v\) for some \(b \in B\) and \(v \in V\); likewise, every \(v \in V\) can be written as \(v=b+w\) for some \(b \in B\) and \(w \in W\). So \(V = B + W\).
\end{proof}

\subsection{Dominant Maps to Varieties of General Type}\label{DomRatMaps}

\begin{theorem}[\cite{KobOch}, Theorem 1]\label{KobOch}
Suppose that \(Y\) is a complex projective variety of general type and that \(X\) is a complex projective variety. Then there are at most finitely many dominant rational maps from \(X\) to \(Y\).
\end{theorem}

Suppose that \(f\) is a surjective endomorphism of a complex abelian variety \(A\) and that \(B \subseteq A\) is an abelian subvariety satisfying \(\phi_f(B)=B\). Then \(\phi_f\) certainly descends to a surjective endomorphism of \(A/B\) and, since \(f(a+b) = \phi_f(b) + f(a)\) for every \(a \in A\) and \(b \in B\), \(f\) also descends to a surjective endomorphism of \(A/B\).

\begin{proof}[Proof of Theorem \ref{MainThm}]
Set \(B = \Stab^0_A(V)\), write \(V=B+W\) as in Theorem \ref{UenoPlus}, let \(q:A \rightarrow A/B\) be the quotient map, and let \(f_B:A/B \rightarrow A/B\) be the quotient of \(f\). Then \(\dim(q(V))=\kappa(q(V))\) and \(f_B\) is a surjective endomorphism of \(q(V)\). Thus, by Theorem \ref{KobOch}, some iterate \(f_B^k\) is the identity map on \(q(V)\). So for any \(w \in W\), \(q(f^k(y))=f_B^k(q(w))=q(w)\) and \(f^k(B+w)=B+w\).
\end{proof}

\begin{remark}
It can be impossible to choose \(W\) in Theorem \ref{MainThm} such that \(f(W)=W\) (even when \(f\) is an isogeny): let \(A\) be an abelian surface, let \(C\) be a smooth hyperbolic curve in \(A\), and let \(f:A \times A \rightarrow A \times A\) be given by \(f(a_1,a_2)=(a_1+a_2,a_2)\); then \(f(A \times C)=A \times C\), but no subset of \(A \times C\) that maps finitely onto \(C\) under the second projection of \(A \times A\) can be preserved by \(f\).
\end{remark}

\subsection{Eigenvalues for Actions on First Cohomology Groups}\label{Eigenvals}\hspace*{\fill}\vspace{6pt}

Given a complex abelian variety \(A\) of dimension \(n\), we can write \(A=\mathbb{C}^n/\Lambda\) for some rank-\(2n\) lattice \(\Lambda \subseteq \mathbb{C}^n\). Then the set \(\{dz_1,\dots,dz_n\}\) of holomorphic 1-forms on \(\mathbb{C}^n\) descends to a basis for \(H^{1,0}(A)\). Since any translation on \(A\) induces a trivial cohomological action, any surjective endomorphism \(f:A \rightarrow A\) satisfies \(f^*=\phi_f^*\) on every cohomology group of \(A\). Moreover, given such an endomorphism, there is some \(\Phi_f \in \GL_n(\mathbb{C})\) satisfying \(\Phi_f(\Lambda)=\Lambda\) such that \(\phi_f\) is the quotient of \(\Phi_f\)--so that \(\phi_f^*=\Phi_f^T\) on \(H^{1,0}(A)\).

\begin{lemma}\label{EigenLemma}
Let \(A\) be a complex abelian variety of dimension \(n\) with an isogeny \(\phi:A \rightarrow A\), and let \(\Gamma = \{\gamma_1,\dots,\gamma_n\}\) be the multiset of eigenvalues of \(\phi^*\) on \(H^{1,0}(A)\). Suppose that \(B \subseteq A\) is an abelian subvariety of dimension \(m\) satisfying \(\phi(B)=B\), and let \(\Delta = \{\delta_1,\dots,\delta_m\}\) be the multiset of eigenvalues of \(\phi^*\) on \(H^{1,0}(B)\). Then \(\Delta\) is a submultiset of \(\Gamma\) and the multiset of eigenvalues of \(\phi_B^*\) on \(H^{1,0}(A/B)\) is \(\Gamma - \Delta\), where \(\phi_B:A/B \rightarrow A/B\) is the quotient of \(\phi\).
\end{lemma}

\begin{proof}
Write \(A = \mathbb{C}^n/\Lambda\)--so that \(\phi\) is the quotient of some \(\Phi \in \GL_n(\mathbb{C})\) satisfying \(\Phi(\Lambda)=\Lambda\). Let \(\pi:\mathbb{C}^n \rightarrow A\) be the quotient map, and let \(V_B = \pi^{-1}(B) \subseteq \mathbb{C}^n\). So \(V_B\) is an \(m\)-dimensional subspace, \(\Lambda_B = V_B \cap \Lambda \subseteq \Lambda\) is a sublattice of rank \(2m\), and \(\Lambda/\Lambda_B\) is a lattice of rank \(2(n-m)\). Let \(q\) be the quotient map from \(A\) to \(A/B\), let \(\rho\) be the quotient map from \(\mathbb{C}^n\) to \(\mathbb{C}^n/V_B\), and let \(\pi'\) be the quotient map from \(\mathbb{C}^n/V_B\) to \((\mathbb{C}^n/V_B)/(\Lambda/\Lambda_B)\). Then \(q \circ \pi\) and \(\pi' \circ \rho\) have the same kernel (i.e., \(V_B \oplus \Lambda\))--so that \((\mathbb{C}^n/V_B)/(\Lambda/\Lambda_B) = A/B\). Now \(\Gamma\) is the multiset of eigenvalues of \(\Phi\), \(V_B\) is \(\Phi\)-invariant, and \(\Delta\) is the multiset of eigenvalues of \(\Phi|_{V_B}\). Moreover, \(\Phi\) descends to a map \(\Phi_B \in \operatorname{GL}(\mathbb{C}^n/V_B)\) and the multiset of eigenvalues of \(\phi_B^*\) on \(H^{1,0}(A/B)\) is the multiset of eigenvalues of \(\Phi_B\).

Let \(\{x_1,\dots,x_m\}\) be a basis for \(V_B\) that gives \(\Phi|_{V_B}\) in Jordan canonical form, and let \(\{y_1,\dots,y_{n-m}\}\) be a subset of \(\mathbb{C}^n\) whose image under \(\rho\) is basis for \(\mathbb{C}^n/V_B\) that gives \(\Phi_B\) in Jordan canonical form. Then \(\{x_1,\dots,x_m,y_1,\dots,y_{n-m}\}\) is a basis for \(\mathbb{C}^n\) with respect to which \(\Phi\) is upper triangular. It follows that the multiset of eigenvalues of \(\Phi_B\) is \(\Gamma-\Delta\).
\end{proof}

Suppose that \(f\) is a surjective endomorphism of a complex abelian variety \(A\) satisfying \(f(\sigma)=\sigma\) for some \(\sigma \in A\). Then
\[\phi_f(a) = f(a + \sigma) - \sigma\]
for every \(a \in A\)--so that \(f\) is conjugate (by a translation) to an isogeny. Moreover, \(\phi_f(V-\sigma)=V-\sigma\) for any \(V \subseteq A\) satisfying \(f(V)=V\).

\begin{proof}[Proof of Corollary \ref{PinRoeCor}]
Set \(B = \Stab^0_A(V)\), write \(V=B+W\) as in Theorem \ref{MainThm}, let \(q:A \rightarrow A/B\) be the quotient map, and let \(f_B:A/B \rightarrow A/B\) be the quotient of \(f\). Then \(f_B^k\) is the identity map on \(q(W)\) and hence, in particular, fixes some point \(\sigma \in A/B\). Thus \(\phi_{f_B^k}\) is the identity map on \(q(W)-\sigma\), and hence also on the abelian subvariety \(T \subseteq A/B\) generated by \(q(W)-\sigma\). (See, e.g., \cite{Deb}, \S 8.) It then follows from Lemma \ref{EigenLemma} that the number of eigenvalues (counting multiplicity) of \((f^k)^*\) on \(H^{1,0}(A)\) that are equal to one is at least \(\dim(T)\), which is at least \(\dim(q(W))=\kappa(V)\). Moreover, if \(\dim(q(W)) \neq 0\), then \(\dim(T)\) must be strictly larger than \(\dim(q(W))\). The proof is concluded by the observation that the multiset of eigenvalues of \((f^k)^*\) on \(H^{1,0}(A)\) is exactly the set of all \(k\)-th powers of elements in the multiset of eigenvalues of \(f^*\) on \(H^{1,0}(A)\).
\end{proof}

\subsection{Unity-Free Endomorphisms of Abelian Varieties}\label{UnityFree1}\hspace*{\fill}\vspace{6pt}

The cohomology ring \(H^*(A,\mathbb{Z})\) of any complex abelian variety \(A\) is generated (via the cup product) by \(H^1(A,\mathbb{Z})\). If \(f\) is a surjective endomorphism of a complex abelian variety \(A\), then the pull-back action \(f^*\) respects the cup product--so that, in particular, the Lefschetz number for \(f\) is
\[\sum_{0 \leq i \leq 2\dim(A)}(-1)^i \Tr(f^*:H^i(A,\mathbb{Z}) \rightarrow H^i(A,\mathbb{Z})) = \prod_{1 \leq j \leq \dim(A)}(1-\gamma_j)(1-\overline{\gamma_j}),\]
where \(\{\gamma_1,\dots,\gamma_{\dim(A)}\}\) is the multiset of eigenvalues of \(f^*\) on \(H^{1,0}(A)\); it then follows from the Lefschetz Fixed-Point Theorem that \(f\) has a fixed point if it is unity-free. (Compare \cite{Zha}, \S 2.1.) Thus any unity-free surjective endomorphism of a complex abelian variety can, without loss of generality, be viewed as an isogeny.

Given a complex abelian variety \(A\), we let \(\Tors(A)\) denote the set of torsion points on \(A\). Given an endomorphism \(f\) of a complex projective variety \(X\), we let \(\Preper(f)\) denote the set of points on \(X\) that are preperiodic for \(f\). If \(f:A \rightarrow A\) is an isogeny of a complex abelian variety \(A\), then \(\Tors(A) \subseteq \Preper(f)\): for any \(m \in \mathbb{N}\), \(\{a \in A \mid ma=0\}\) is finite and preserved by \(f\). The following result gives a useful characterization of unity-free endomorphisms of abelian varieties.

\begin{proposition}\label{MainProp}
Let \(f:A \rightarrow A\) be an isogeny of a complex abelian variety \(A\). Then the following three conditions are equivalent:
\begin{itemize}
\item[1)] \(\Preper(f) \neq \Tors(A)\); 
\item[2)] there is a positive-dimensional abelian subvariety of \(A\) that is pointwise fixed by some iterate \(f^k\); and 
\item[3)] there is an eigenvalue of \(f^*\) on \(H^{1,0}(A)\) that is a root of unity.
\end{itemize}
\end{proposition}

\begin{proof}
(1 $\Rightarrow$ 2) If ${\rm Preper}(f) \ne {\rm Tors}(A)$, then there is a nontorsion point $P \in A$ satisfying a preperiodic relation; i.e., there exists $m > n \in \mathbb{N}_0$ such that $f^m(P) = f^n(P)$. Since $f$ is an isogeny, any iterate $f^k(P)$ is also nontorsion; so, without loss of generality, we can take $P$ to be periodic and set $n=0$. Since $\{a \in A : f^k(a) = a\}$ is a group subvariety of $A$ containing $mP$ for any $m \in \mathbb{Z}$, it must contain some positive-dimensional abelian subvariety of $A$ that is pointwise fixed by $f^k$. 

(2 $\Rightarrow$ 3) If $f^k$ pointwise fixes some positive-dimensional abelian subvariety $K \subseteq A$, then the eigenvalues of \((f^k)^*\) on \(H^{1,0}(K)\) must all be one. So it follows from Lemma \ref{EigenLemma} that the eigenvalues of \(f^*\) on \(H^{1,0}(A)\) must include a root of unity.

(3 $\Rightarrow$ 1) Write $A = \mathbb{C}^{\dim(A)}/\Lambda$--so that \(f\) is the quotient of some \(F \in \GL_{\dim(A)}(\mathbb{C})\) satisfying \(F(\Lambda)=\Lambda\). If some eigenvalue of $f^*$ on $H^{1,0}(A)$ is a root of unity, then 1 is an eigenvalue of some iterate $(f^*)^k$ on $H^{1,0}(A)$--and hence also of $F^k$. Let $(g_{ij})_{1 \leq i,j \leq \dim(A)}$ be a matrix representation of $F^k$ as a linear self-map on $\mathbb{C}^{\dim(A)}$. Under the natural identification of $\mathbb{C}^{\dim(A)}$ with $\mathbb{R}^{2\dim(A)}$ via
$$z_l = x_l + \imath y_l \mapsto (x_l,y_l),$$
this matrix represention of $F^k$ becomes
$$\left( \begin{array}{cc}
\Re (g_{ij}) & -\Im (g_{ij}) \\
\Im (g_{ij}) & \Re (g_{ij}) \end{array} \right)_{1 \leq i,j \leq \dim(A)}.$$
Taking $(g_{ij})_{1 \leq i,j \leq \dim(A)}$ in Jordan canonical form shows immediately that 1 is an eigenvalue of $F^k$ on $\mathbb{R}^{2\dim(A)}$--and hence also on ${\rm Span}_\mathbb{Q}(\Lambda)$. So $F^k$ pointwise fixes some non-trivial linear subspace $V \subseteq {\rm Span}_\mathbb{Q}(\Lambda)$, and $f^k$ pointwise fixes the non-trivial (real) subtorus $T \subseteq A$ corresponding to the closure of $V$ in $\mathbb{R}^{2\dim(A)}$. Thus ${\rm Preper}(F)$ contains all of $T$, including its nontorsion points.
\end{proof}

\begin{remark}\label{TorsTran}
In Theorem \ref{PinRoe}, the hypothesis that \(f\) is an isogeny actually leads to a stronger conclusion, via an application of Proposition \ref{MainProp}: set \(B=\Stab_A^0(V)\), write \(V=B+w\) as in Theorem \ref{MainThm} with \(w \in A\) a point, let \(q:A \rightarrow A/B\) be the quotient map, and let \(f_B:A/B \rightarrow A/B\) be the quotient of \(f\); then, since \(f_B\) is an isogeny and \(f_B(q(w))=q(w)\), \(q(w)\) must be an element of \(\Tors(A/B)\); thus \(w\) can be taken to be an element of \(\Tors(A)\)--so that \(V\) is in fact a torsion translate of an abelian subvariety.
\end{remark}

\section{Induced Maps on Albanese Varieties}

Given a smooth complex projective variety \(X\), we let \(\Alb(X)\) denote the Albanese variety for \(X\) and we let \(\alpha_X\) denote the Albanese map from \(X\) to \(\Alb(X)\).

\subsection{Functorial Properties of Albanese Maps}\label{AlbIntro}\hspace*{\fill}\vspace{6pt}

Any endomorphism \(f\) of a smooth complex projective variety \(X\) induces a map \(F:\Alb(X) \rightarrow \Alb(X)\) satisfying \(F \circ \alpha_X = \alpha_X \circ f\); moreover, since \(\alpha_X^*\) gives an isomorphism from \(H^{1,0}(\Alb(X))\) to \(H^{1,0}(X)\), the pull-back action \(F^*\) on \(H^{1,0}(\Alb(X))\) is conjugate to the pull-back action \(f^*\) on \(H^{1,0}(X)\). (See, e.g., \cite{Huy}, \S 3.3.) The universal property of Albanese varieties states that any morphism from \(X\) to a complex abelian variety must factor through \(\alpha_X\)--so that, in particular, \(\alpha_X(X)\) cannot be contained in a translate of a proper abelian subvariety of \(A\). (See, e.g., \cite{BHPV}, \S I.13.) So \(F\) must be surjective if \(f\) is surjective.

Suppose now that \(\alpha_X(X) \neq \Alb(X)\) and that \(f\) is surjective. Then Theorem \ref{MainThm} shows that
\[\alpha_X(X) = \Stab_{\Alb(X)}^0(\alpha_X(X)) + W\]
for some variety \(W \subseteq \Alb(X)\) of general type, and that there is an iterate \(F^k\) that satisfies
\[F^k(\Stab_{\Alb(X)}^0(\alpha_X(X)) + w)=\Stab_{\Alb(X)}^0(\alpha_X(X)) + w\] for every \(w \in W\). Moreover, the universal property of Albanese varieties implies that \(\kappa(\alpha_X(X)) > 0\), and hence also that \(\kappa(W) > 0\).

\subsection{Endomorphisms of Varieties with Non-Surjective Albanese Maps}\label{AlbCors}

\begin{proof}[Proof of Corollary \ref{PerSubCor}]
Set \(B=\Stab_{\Alb(X)}^0(\alpha_X(X))\), write \(\alpha_X(X) = B + W\) as in Theorem \ref{MainThm}, and let \(F:\Alb(X) \rightarrow \Alb(X)\) be the map induced by \(f\). Since \(\dim(W)>0\) and \(F^k(B+w)=B+w\) for any \(w \in W\), the pre-image of \(B+w\) satisfies \(\alpha_X^{-1}(B+w) \neq X\) and
\[f^k(\alpha_X^{-1}(B+w))=\alpha_X^{-1}(B+w)\]
for any \(w \in W\). If \(\dim(B)>0\) then \(\alpha_X^{-1}(B+w)\) is a proper positive-dimensional subvariety of \(X\) for any \(w \in W\).

Suppose that \(\dim(B)=0\). If it were the case that \(\dim(\alpha_X^{-1}(B+w))=0\) for some \(w \in W\), then \(\alpha_X:X \rightarrow \alpha_X(X)\) would necessarily be a generically finite map; but then \(X\) would necessarily have the same dimension as \(\alpha_X(X)\) and hence (as in the argument that \(W\) is a variety of general type in the proof of Theorem \ref{UenoPlus}) would be a variety of general type--which, by Theorem \ref{KobOch}, contradicts the assumption that \(X\) admits an infite-order surjective endomorphism. So \(\alpha_X^{-1}(B+w)\) is still a proper positive-dimensional subvariety of \(X\) for any \(w \in W\).
\end{proof}

\begin{proof}[Proof of Corollary \ref{PolEndCor}]
Since \(\kappa(\alpha_X(X))>0\), Corollary \ref{PinRoeCor} shows that no endomorphism of \(\Alb(X)\) that is induced by a surjective endomorphism of \(X\) can be unity-free. So no surjective endomorphism of \(X\) can be unity-free.
\end{proof}

\begin{theorem}[\cite{Ser}, Th\'eor\`eme 1]\label{Serre}
Suppose that \(f\) is a polarized endomorphism of a smooth complex projective variety \(X\) of dimension \(n\) and that, in particular, \(f^*L=L^{\otimes q}\) for some ample \(L \in \Pic(X)\) and \(q \in \mathbb{N}-\{1\}\). Then, for every \(j \in \{0,\dots,2n\}\), the magnitude of every eigenvalue of \(f^*\) on \(H^j(X,\mathbb{Z})\) is \(q^{j/2}\).
\end{theorem}

It is clear from Theorem \ref{Serre} that a polarized endomorphism must be unity-free if it occurs on a variety whose Albanese map is non-trivial - and hence cannot occur on a variety whose Albanese map is non-surjective. For a surjective endomorphism \(f\) of a smooth complex projective variety \(X\), the \(j\)-th dynamical degree of \(f\) is
\[\lambda_j(f) = \rho(f^*:H^{2j}(X,\mathbb{Z}) \rightarrow H^{2j}(X,\mathbb{Z})),\]
where \(\rho\) denotes the spectral radius. It is again clear from Theorem \ref{Serre} that polarized endomorphisms are excluded (because they have distinct consecutive dynamical degrees) from varieties with non-surjective Albanese maps by the following result.

\begin{theorem}[\cite{DinNguTru}, Corollary 1.4] \label{DNT}
Let \(f\) be a surjective endomorphism of a smooth complex projective variety \(X\) of dimension \(n\), and suppose that \(\alpha_X(X) \neq \Alb(X)\). Then there is an integer \(j \in \{0,\dots,n-1\}\) such that \(\lambda_j(f)=\lambda_{j+1}(f)\).
\end{theorem}

We remark that Theorem \ref{DNT} is independent from Corollary \ref{PolEndCor}; that is, there exist endomorphisms which are not unity-free but have distinct dynamical degrees, and there exist endomorphisms which are unity-free but do not have distinct dynamical degrees.  An example of the former is simply the multiplication map $[2] \times [1]$ on $E \times E$ for any elliptic curve $E$. On the other hand, the automorphism (among others) of \(E \times E \times E \times E\) given by
\[(e_1,e_2,e_3,e_4) \mapsto (e_2,e_3,e_4,-e_1+3e_2+4e_3+3e_4)\]
is unity-free but does not have distinct consecutive dynamical degrees. (See also \cite{OguTru}.)

The constraint on varieties admitting polarized endomorphisms provided by Corollary \ref{PolEndCor} complements the following characterization.

\begin{theorem}[\cite{Fak}, Theorem 4.2]\label{Fakhruddin}
Let \(X\) be a smooth complex projective variety admitting a polarized endomorphism, and suppose that \(\kappa(X) \geq 0\). Then there is an abelian variety \(A\) and a finite surjective map \(\pi:A \rightarrow X\).
\end{theorem}

Since the Kodaira dimension of any abelian variety is zero, it follows from Theorem \ref{Fakhruddin} that the Kodaira dimension of any smooth complex projective variety admitting a polarized endomorphism must be non-positive. Corollary \ref{PolEndCor} addresses the case of a smooth complex projective variety with negative Kodaira dimension and a non-surjective Albanese map. For example, if \(X \rightarrow Y\) is a fiber bundle whose fibers have negative Kodaira dimension and whose base satisfies \(\alpha_Y(Y) \neq \Alb(Y)\), then \(\kappa(X)<0\) and (by the universal property of Albanese varieties) \(\alpha_X(X) \neq \Alb(X)\).

\subsection{Implications for a Dynamical Manin-Mumford Conjecture}\label{DMM}

The Manin-Mumford Conjecture (proved by Raynaud) states that a reduced and irreducible subvariety \(V \subseteq A\) of a complex abelian variety \(A\) is a torsion translate of an abelian subvariety if and only if \(V \cap \Tors(A)\) is Zariski dense in \(V\). (See, e.g., \cite{Ray} and \cite{GhiTucZha}.) The following conjecture (now known to be false) is a first attempt to transport this idea to dynamical systems.

\begin{conjecture}[\cite{Zha}, Conjecture 1.2.1]\label{DMMConj}
Let \(f\) be a polarized endomorphism of a smooth complex projective variety \(X\), and let \(Y \subseteq X\) be a reduced and irreducible subvariety. Then \(Y\) is preperiodic for \(f\) if and only if \(Y \cap \Preper(f)\) is Zariski dense in \(Y\).
\end{conjecture}

When \(X\) is an abelian variety, Proposition \ref{MainProp} constrains the subvarieties \(Y\) which could disprove Conjecture \ref{DMMConj} by containing Zariski dense sets of preperiodic points without themselves being preperiodic: since, by Theorem \ref{Serre}, \(f\) is unity-free, it follows from Proposition \ref{MainProp} that \(\Preper(f)=\Tors(X)\); thus, by the Manin-Mumford Conjecture, any reduced and irreducible subvariety \(Y \subseteq X\) with \(Y \cap \Preper(f)\) Zariski dense in \(Y\) must be a torsion translate of an abelian subvariety of \(X\). We note below that Conjecture \ref{DMMConj} does in fact fail in this direction--and the main counterexamples are indeed torsion translates of abelian subvarieties which are not preperiodic. As for the converse direction of Conjecture \ref{DMMConj}, the following results show that it is true when \(X\) is an abelian variety even when the requirement that \(f\) be polarized is relaxed to require only that \(f\) be unity-free.

\begin{proposition}\label{AbVarDMM}
Let \(f:A \rightarrow A\) be a unity-free isogeny of a complex abelian variety \(A\), and suppose that \(V \subseteq A\) is a reduced and irreducible subvariety that is preperiodic for \(f\). Then \(V \cap \Preper(f)\) is Zariski dense in \(V\). Moreover, \(V\) is a torsion translate of an abelian subvariety of \(A\).
\end{proposition}

\begin{proof}
There are iterates \(f^{k_1}\) and \(f^{k_2}\) such that \(f^{k_1}(f^{k_2}(V))=f^{k_2}(V)\). Since \(f^{k_1}\) is a unity-free isogeny, Remark \ref{TorsTran} shows that \(f^{k_2}(V)\) must be a torsion translate of an abelian subvariety of \(A\). So, by the Manin-Mumford Conjecture, the set \(P = f^{k_2}(V) \cap \Tors(A)\) must be Zariski dense in \(f^{k_2}(V)\). Then \((f^{k_1})^{-1}(P) \cap V\) is Zariski dense in \(V\) and consists entirely of points in \(\Preper(f)\). Since \(\Preper(f)=\Tors(A)\), \(V\) must itself be a torsion translate of an abelian subvariety of \(A\).
\end{proof}

If \(V\) (with \(\dim(V)>0\)) is periodic, so fixed by some iterate \(f^k\), in Proposition \ref{AbVarDMM}, then \(f^k|_V\) is again a unity-free isogeny: there is some \(\tau \in V \cap \Tors(A)\) that is fixed by some iterate \(f^{k'}\) with \(k|k'\); so \(V'=V-\tau\) is an abelian subvariety of \(A\) satisfying \(f^{k'}(V')=V'\); it follows from Lemma \ref{EigenLemma} that \(f^{k'}\) is unity-free on \(V'\); finally, since \(f^{k'}|_{V'}\) is conjugate to \(f^{k'}|_V\), \(f^{k'}\) (and hence also \(f^k\)) must be unity-free on \(V\). By the following result, we conclude that \(V\) in fact contains a Zariski dense set of periodic points.

\begin{proposition}\label{AbVarZDense}
Let \(f:A \rightarrow A\) be a unity-free isogeny of a complex abelian variety \(A\). Then the set of periodic points for \(f\) is Zariski dense in \(A\).
\end{proposition}

\begin{proof}
Let \(B \subseteq A\) be the Zariski closure of the periodic points for \(f\), and let \(B^0\) be an irreducible component of \(B\) containing the identity. Since \(f(B)=B\), every irreducible component of \(B\) is preperiodic, and hence is a torsion translate of an abelian subvariety of \(A\). If \(B'\) is an irreducible component of \(B\) containing the identity, then every point of the form \(\tau+\tau'\) with \(\tau \in B^0\) periodic and \(\tau' \in B'\) periodic is also periodic and the set of all such points is Zariski dense in \(B^0 + B'\); so \(B'=B^0\). Since the identity is a fixed point for \(f\), it follows that \(f(B^0)=B^0\). If \(B'\) is any irreducible component of \(B\) and \(\tau' \in B'\) is periodic, then \(B'-\tau'\) is contained in \(B^0\). Let \(q:A \rightarrow A/B^0\) be the quotient map, and let \(f_B:A/B^0 \rightarrow A/B^0\) be the quotient of \(f\). So the image of \(B\) in \(A/B^0\) is a finite set of points. If \(\sigma \in A/B^0\) is periodic for \(f_B\), then the orbit of \(q^{-1}(\sigma)\) is finite under \(f\) and some component of \(q^{-1}(\sigma)\) is periodic under \(f\); so \(q^{-1}(\sigma) \cap B \neq \emptyset\) and \(\sigma\) is in the image of \(B\). Thus (by Lemma \ref{EigenLemma}) \(f_B\) is a unity-free isogeny with only finitely many periodic points if \(\dim(A/B^0)>0\), which cannot happen. So \(B^0=A\).
\end{proof}

As shown by Ghioca, Tucker, and Zhang \cite{GhiTucZha}, counterexamples to Conjecture \ref{DMMConj} can be constructed on an abelian surface of the form \(E \times E\), where \(E\) is an elliptic curve with complex multiplication: an endomorphism given by coordinate-wise multiplication by distinct elements of \(\End(E)\) with the same magnitude will always be polarized, but may give the diagonal in \(E \times E\) an infinite orbit; on the other hand, the diagonal in \(E \times E\) will always contain infinitely many torsion points, all of which must be preperiodic for the endomorphism. For additional details, see \cite{GhiTucZha}, \S 2. Many similar examples can also be constructed on higher-dimensional abelian varieties; see \cite{Paz}. All of the known counterexamples to Conjecture \ref{DMMConj} come from this type of construction, and attempts have been made to modify Conjecture \ref{DMMConj} to accommodate these examples. Ghioca, Tucker, and Zhang offer the following modification:

\begin{conjecture}[\cite{GhiTucZha}]\label{GTZconj} Let $X$ be a projective variety, $f:X \rightarrow X$ a polarized endomorphism defined over $\mathbb{C}$, and $Y$ a subvariety with no component contained in the singular part of $X$.  Then $Y$ is preperiodic under $f$ if and only if there exists a subset of smooth preperiodic points $x \in Y$ which are Zariski dense in $Y$, such that the tangent subspace of $Y$ at $x$ is preperiodic under the induced action of $f$ on the Grassmanian Gr$_{dim(Y)}(T_{X,x})$.  
\end{conjecture}

In \cite{GhiTucZha}, Conjecture \ref{GTZconj} is verified for group endomorphisms of abelian varieties, and for the case $X = \mathbb{P}^1 \times \mathbb{P}^1$, $Y$ a line, and $f$ a product map. It is worth noting that the tangent space condition is essentially used only to eliminate counterexamples of the form mentioned above, though they can appear subtly in the form of Latt\'es maps.

\begin{remark}\label{UFDMM}
In light of Propositions \ref{AbVarDMM} and \ref{AbVarZDense}, it is natural to ask if the assumption in Conjecture \ref{DMMConj} that \(f\) is polarized should be replaced by the assumption that \(f\) is unity-free (along with whatever other changes are made to account for the known counterexamples). However, outside the realm of abelian varieties, a unity-free endomorphism can fail to have a Zariski dense set of preperiodic points: the endomorphism of \(\mathbb{P}^1 \times E\) (where \(E\) is an elliptic curve) given by
\[([x_0:x_1],e) \mapsto ([2x_0:x_1],2e)\]
is unity-free, but has all of its preperiodic points contained in \(\{[0,1],[1,0]\} \times E\). Moreover, it is possible in general for a unity-free endomorphism to have an invariant subvariety on which the restriction is not unity-free.
\end{remark}

\section{Cohomological Properties of Endomorphisms}

\subsection{Properties of Polarized and Amplified Endomorphisms}\label{BasicProps}

We place the cohomological conditions of the above theorems into context among other endomorphisms of projective varieties. 

\begin{proposition}\label{Iterates} Let $X$ be a smooth projective variety over $\mathbb{C}$, and $\phi: X \rightarrow X$ a surjective endomorphism.  The following hold:
\begin{enumerate}
\item If $f$ is polarized (resp. amplified or unity-free), $f^k$ is polarized (resp. amplified or unity-free) for all $k \geq 1$.  \\
\item If $f$ is polarized (resp. amplified) and $Y$ is a closed subvariety of $X$ satisfying $f(Y) = Y$, then $f \mid_Y$ is polarized (resp. amplified). \\
\item If $f$ is amplified, each set of the form $\{ x \in X : f^m(x) = f^n(x) \}$ for $m > n \geq 0$ is a finite set.
\end{enumerate}
\end{proposition}

\begin{proof} The first two statements are easily checked, since the restriction of an ample divisor to a closed subvariety is ample, and restriction commutes with the action of $f^*$.  For the third, note that if $m>n \geq 0$ and $\{ x \in X : f^m(x) = f^n(x) \}$ is not finite, then it is a closed, positive-dimensional subvariety $Y$ of $X$, and $Z = f^n(Y)$ is pointwise-fixed by $f^{m-n}$.  Since $f$ is amplified, $g := f^{m-n} \mid_Z$ is amplified, so there exists a line bundle $L$ on $Z$ such that $g^*(L) \otimes L^{-1}$ is ample.  However, $g$ acts trivially on $Z$; so we conclude that the trivial bundle on $Z$ is ample, a contradiction.
\end{proof}

When an endomorphism $f$ is not amplified, there is an immediate consequence for the action of \(f^*\) on the N\'eron-Severi group \(\NS(X)\): since the linear transformation \(f^*-ID\) cannot be surjective on \(\NS(X)_\mathbb{Q}\) (as it must miss the ample cone), 1 must be an eigenvalue of \(f^*\) on \(\NS(X)_\mathbb{Q}\); so there must be some line bundle in \(\Pic(X)\) has numerical equivalence class is fixed by \(f^*\).

As discussed in \S \ref{DMM}, the simplest proposed version of a dynamical Manin-Mumford conjecture was proven to be false in \cite{GhiTucZha}, and an alternate conjecture proposed. Both versions include the strong hypothesis that the endomorphism $f: X \rightarrow X$ be polarized.  However, for the direction of the conjecture which is true, this hypothesis is unnecessarily strong, as was shown by Fakhruddin. 

\begin{theorem} [\cite{Fak}] \label{Fak} Let $X$ be a projective variety over an algebraically closed field, and $f:X \rightarrow X$ a dominant amplified morphism.   Then the subset of $X$ consisting of periodic points is Zariski dense in $X$.
\end{theorem}

As noted in Remark \ref{UFDMM}, unity-free endomorphisms are likely not the right setting for a dynamical Manin-Mumford conjecture; however, Fakhruddin's theorem gives hope that a dynamical Manin-Mumford conjecture may hold in the much broader setting of amplified endomorphisms.

\subsection{The Implication Diagram for Varieties with Non-Trivial Albanese Maps}\label{Diagram}

\begin{proof}[Proof of Theorem \ref{EndoThm}] Suppose $X$ is a smooth complex projective variety with non-trivial Albanese, $f: X \rightarrow X$ is a dominant, amplified endomorphism, and $f$ is not unity-free. By Proposition \ref{Iterates}, these conditions will hold for any iterate of $f$ as well.  By Theorem \ref{Fak}, some iterate of $f$ has a fixed point.  Replacing $f$ by this iterate, $f$ has a fixed point, and so the Albanese map can be chosen so that the induced map $F: \text{Alb}(X) \rightarrow \text{Alb}(X)$ is an isogeny.

Since $f$ is not unity-free, $F$ is not unity-free.  By Proposition \ref{MainProp}, $\text{Alb}(X)$ contains a positive-dimensional abelian subvariety $T$ which is pointwise fixed by some iterate of $f$.  Replace $f$ by an iterate to assume $T$ is pointwise fixed by $F$.   Since $\alpha_X: X \rightarrow \text{Alb}(X)$ has image which generates $\text{Alb}(X)$ as a group, there exists a positive integer $M$ such that the map 
$$\alpha_M: X^{\times 2M} \rightarrow \text{Alb}(X)$$
given by
$$\alpha_M(x_1, \dots, x_{2M}) := \alpha_X(x_1) + \cdots + \alpha_X(x_M) - \alpha_X(x_{M+1}) - \cdots - \alpha_X(x_{2M})$$
satisfies $T \subset \alpha_M(X).$  Let $f_{2M} : X^{\times 2M} \rightarrow X^{\times 2M}$ denote the coordinate-wise application of $f$ to $X^{\times 2M}$.  Since $f$ is amplified, there exists a line bundle $L \in \Pic(X)$ with $f^*(L) \otimes L^{-1}$ ample; then $f_{2M}$ is amplified with respect to the bundle $\pi_1^*(L) \otimes \cdots \otimes \pi_{2M}^*(L)$, where $\pi_j$ is the usual projection to the $j$th coordinate.  By definition, $f_{2M}$ fixes the fiber $S_t$ over any point $t \in T$; since $f_{2M}$ is amplified, Proposition \ref{Iterates} and Theorem \ref{Fak} imply that each fiber $S_t$ contains a Zariski dense subset of periodic points.  Since there are uncountably many such fibers, there must be some positive integer $N$ such that infinitely many points in $X^{\times 2M}$ have exact period $N$.  Therefore $X^{\times 2M}$ contains a positive-dimensional subvariety which is pointwise fixed by $f_{2M}^N$, which is a contradiction by Proposition \ref{Iterates}. 
\end{proof}

\subsection{Converse Directions in the Implication Diagram}\label{RevImp}

We now make some remarks on the relative strengths of the various types of endomorphisms defined throughout.  By Theorem \ref{EndoThm}, we have the following diagram for any endomorphism of a smooth complex projective variety $X$ with non-trivial Albanese:
\vspace{.05in}
\[
\text{polarized} \implies \text{amplified} \implies \text{unity-free} \implies \text{infinite-order}.
\]

In general, none of the reverse implications are true; we provide examples from the right-hand side of the diagram to the left.  By Proposition \ref{MainProp}, the product of any infinite-order endomorphism on an abelian variety $A$ with the identity map $id_A$ will have infinite-order, but not be unity-free.  By Remark \ref{UFDMM}, there exist unity-free endomorphisms whose periodic points are contained in a proper subvariety; by Theorem \ref{Fak}, such an endomorphism cannot be amplified.  Finally, consider the map $\phi := \left[ 2 \right] \times \left[ 3 \right]$ on the product $E \times E$ of an elliptic curve $E$.  The eigenvalues of $\phi^*$ on $H^{1,1}(E \times E)$ are 4, 6, and 9. So, by Theorem \ref{Serre}, $\phi$ is not polarized; on the other hand, since 1 is not an eigenvalue of \(f^*\) on \(\NS(X)\), \(\phi\) is amplified.

Note that the above counterexamples to the reverse implications were given for {\it abelian} varieties, except for the unity-free, non-amplified example.  Additionally, these counterexamples could occur in all dimensions $\geq 2$.  It is perhaps surprising then that for abelian surfaces, unity-free does imply amplified.

\begin{proposition} \label{abeliansurfaces}
Let $f$ be a surjective endomorphism of an abelian surface $X$ which is not amplified.  Then $f$ is not unity-free.
\end{proposition}

\begin{proof} Let $\gamma_1$ and $\gamma_2$ be the eigenvalues of $f^*$ on $H^{1,0}(X)$. Then the eigenvalues of $f^*$ on $H^{0,1}(X)$ are $\overline{\gamma_1}$ and $\overline{\gamma_2}$, and the eigenvalues of $f^*$ on $H^{1,1}(X)$ are $|\gamma_1|^2$, $\gamma_1 \overline{\gamma_2}$, $\overline{\gamma_1} \gamma_2$, and $|\gamma_2|^2$. Since \(f\) is not amplified, 1 is an eigenvalue of \(f^*\) on \(H^{1,1}(X)\).

If $|\gamma_1|=1$, then $\gamma_1^{-1}$ is a Galois conjugate of $\gamma_1$. So the minimal polynomial for $\gamma_1$, which is a factor in the characteristic polynomial for $f^*$ on $H^1(X,\mathbb{Z})$, is reciprocal--and hence has constant term 1. If $\gamma_2$ is a Galois conjugate of $\gamma_1$, then so is $\gamma_2^{-1}$--which forces $|\gamma_2|=1$. So, whether or not $\gamma_2$ is a Galois conjugate of $\gamma_1$, it follows from Kronecker's theorem that the minimal polynomial for $\gamma_1$ is cyclotomic. (A similar argument applies if $|\gamma_2|=1$.)

If $|\gamma_1| \neq 1$ and $|\gamma_2| \neq 1$, then $\gamma_1 \overline{\gamma_2} = \overline{\gamma_1} \gamma_2 = 1$. So the topological degree of $f$--that is, the eigenvalue of $f^*$ on $H^4(X)$--is $\gamma_1 \gamma_2 \overline{\gamma_1} \overline{\gamma_2} = 1$; thus $f$ is an automorphism. Since $|\gamma_2|=|\gamma_1|^{-1} \neq 1$, $f$ has positive entropy. Since the signature of the intersection form on $H^{1,1}(X)$ is (1,3) and the signature of the subspace of $H^{1,1}(X)$ generated by the eigenvectors for $|\gamma_1|^2$ and $|\gamma_2|^2$ is (1,1), the eigenspace for the eigenvalue 1 must be negative definite. So any periodic curve for $f$ would necessarily have negative self-intersection, and hence by the adjunction formula would necessarily be a rational curve; but abelian varieties cannot contain rational curves. Thus $f$ has no periodic curves, and Lemma \ref{surfautomlemma} below contradicts the assumption that $f$ is not amplified.
\end{proof}

\begin{lemma}\label{surfautomlemma}
Let $f$ be an automorphism with positive entropy of a smooth complex projective surface $X$. If $f$ has no periodic curves, then $f$ is amplified.
\end{lemma}

\begin{proof}
(See \cite{Res} for details.) The entropy of $f$ is $\log(\lambda)$ for some Salem number $\lambda$. There is a sublattice $NS(X)' \leq NS(X)$ such that $f^*$ preserves $NS(X)'$ and the characteristic polynomial here is the minimal polynomial for $\lambda$. So, in particular, $f^*-\ID$ is surjective on $NS(X)'$. Since $f$ has no periodic curves, $NS(X)'$ contains classes of ample line bundles.
\end{proof}

The proof given of Proposition \ref{abeliansurfaces} intrinsically uses the well-understood intersection theory of abelian surfaces.  However, in light of Proposition \ref{AbVarZDense}, it seems possible that all unity-free endomorphisms of an abelian variety are amplified; we leave the question for the reader and future exploration.

Lemma \ref{surfautomlemma} shows that the types of varieties admitting polarized endomorphisms are more restricted than those admitting amplified endomorphisms. Indeed, there are many examples of K3 surface automorphisms with positive entropy and no periodic curves. On the other hand, a K3 surface can never admit a polarized endomorphism. (See, e.g., \cite{FujNak}.)

\bibliographystyle{abbrv}
\bibliography{ReschkeP-refs-2013.08.04}

\end{document}